\documentclass[12pt]{amsart}

\allowdisplaybreaks[1]
\usepackage{enumerate}
\usepackage{graphicx}
\usepackage{allrunes}
\usepackage{amsmath}
\usepackage{amssymb}
\usepackage{amsfonts}
\usepackage[dvipsnames]{xcolor}
\usepackage{hyperref}
\usepackage{skak}
\allowdisplaybreaks
\makeindex

 \newtheorem{theorem}{Theorem}[section]
 \newtheorem{corollary}[theorem]{Corollary}
 \newtheorem{lemma}[theorem]{Lemma}
 \newtheorem{proposition}[theorem]{Proposition}

\theoremstyle{definition}

\theoremstyle{remark}

\newtheorem{fact*}{Fact}

\DeclareMathOperator\im{\mathrm {Im~}}
\newcommand{\conv}{\mathrm{ Conv}}

\renewcommand\h{\mathcal{H}}

\renewcommand{\M}{\mathcal{M}}

\renewcommand{\D}{\mathbb{D}}
\newcommand{\C}{\mathbb{C}}

\newcommand{\BH}{\mathcal{B}(\mathcal{H})}

\renewcommand{\R}{\mathbb{R}}

\newcommand{\cc}[1]{\overline{#1}}
\newcommand{\abs}[1]{\left\vert#1\right\vert}

\newcommand{\norm}[1]{\left\Vert#1\right\Vert}

\newcommand{\ip}[2]{\left\langle #1, #2 \right\rangle}
\newcommand{\ad}{^\ast}
\newcommand{\inv}{^{-1}}

\newcommand{\til}{\raise.17ex\hbox{$\scriptstyle\mathtt{\sim}$}}

\newcommand{\ph}{\varphi}

\newcommand\ep{\varepsilon}
\newcommand\la{\lambda}

\newcommand\beq{\begin{equation}}

\newcommand\eeq{\end{equation}}

\newcommand{\bbm}{\left[ \begin{smallmatrix}}
\newcommand{\ebm}{\end{smallmatrix} \right]}
\newcommand{\bpm}{\left( \begin{smallmatrix}}
\newcommand{\epm}{\end{smallmatrix} \right)}
\numberwithin{equation}{section}

\newlength{\Mheight}
\newlength{\cwidth}

\newcommand{\dfn}[1]{{\bf #1}\index{#1}}

\newcommand{\BK}{\mathcal{B}(\mathcal{K})}
\newcommand{\KK}{\mathcal{K}}

\newcommand{\GG}{\textarc{d}}
\newcommand{\sov}{\textarc{s}}
\newcommand{\funct}{F}

\title[Automatic nc real analyticity]{The royal road to automatic noncommutative real analyticity, monotonicity, and convexity}
\author[J. E. Pascoe]{
J. E. Pascoe$^\symking$
}
\address{Department of Mathematics\\
1400 Stadium Rd\\
  University of Florida\\
 Gainesville, FL 32611}
\email[J. E. Pascoe]{pascoej@ufl.edu}

\thanks{$\symking$ The authors were generously supported by the Fields Institute, Focus Program on Applications of Noncommutative Functions}

\author[R. Tully-Doyle]{
Ryan Tully-Doyle$^\symking$
}
\address{Department of Mathematics and Physics \\
University of New Haven\\
West Haven, CT 06516 }
\email[R. Tully-Doyle]{rtullydoyle@newhaven.edu}

\date{\today}


\begin{document}

\begin{abstract}
It was shown classically that matrix
monotone and matrix convex functions must be real analytic by L\"owner and Kraus respectively.
Recently, various analogues have been found in several noncommuting variables.
We develop a general framework for lifting automatic analyticity theorems in matrix analysis
from one variable to several variables, the so-called ``royal road theorem."
That is, we establish the principle that the hard part of proving any automatic analyticity theorem
lies in proving the one variable theorem.
 We use our main result to prove the noncommutative L\"owner and Kraus theorems over operator systems as examples, including an analogue of the ``butterfly realization" of Helton-McCullough-Vinnikov for general analytic functions.
\end{abstract}

\maketitle

\tableofcontents

\section{Introduction}\label{sec:intro}

There is no royal road to L\"owner's theorem in one variable. However, there is a royal road to the multi-variable L\"owner theorem in noncommutative function theory: the one variable L\"owner theorem itself. (Barry Simon counts 11, or perhaps 12, proofs of the one variable theorem, none of which are regarded as trivial \cite{simonlow}. Thorough treatments are given in \cite{bha97, don74}.)  The purpose of the present quest is to give a general regime for turning one variable theorems in the intersection of classical complex analysis and operator theory into theorems in multiple noncommuting variables using a so-called ``royal road theorem'' built on the absolute and supreme powers of several complex variables and convexity.  We use this ``royal road'' to prove the analogues of the celebrated theorems of L\"owner \cite{lo34} and Kraus \cite{kraus36} in the multivariable setting as mere examples of a very general analytic technique. (The multivariable L\"owner theorem has been established in many settings. In commuting variables, see \cite{amyloew, pascoelownote}. In noncommuting variables, see \cite{palfia, pastdfree}, culminating in essentially the most general framework in \cite{pascoeopsys}, which we reprove here using the ``royal road'' as a shortcut. Convexity theorems are somewhat less generally developed \cite{dhmconvex, hhlm2008, helniesem, helmconvex04, hptdvconvex, Helton2012, palfia}.)

Matthew Kennedy gave a talk at the Fields Institute on Monday, June 10, 2019, on recent work with Kenneth Davidson on noncommutative Choquet theory \cite{kennyd}. Prominent in the theory was the role of the matrix convex function. The merit of matrix convex functions was appreciated essentially on the level of classically convex functions. However, as there is a great gulf between positive and completely positive maps, so too should there be between convex and matrix convex functions, as was first discovered by Kraus \cite{kraus36}. In light of the recent progress with respect to the related topic of matrix monotonicity, it seemed clear here that automatic analyticity should hold, and for reasons arising more from complex analysis and the one variable theorem than an artisanal approach starting from scratch. This provided additional motivation for the current endeavor. 

\subsection{The classical theorems}

Let $f:(a,b) \to \R$ be a function. We say that $f$ is \dfn{matrix monotone} if 
\[
A \leq B \Rightarrow f(A) \leq f(B)
\]
for all $A, B$ self-adjoint of the same size with spectrum in $(a,b)$, where $A \leq B$ means that $B - A$ is positive semidefinite. (The function $f$ is evaluated via the matrix functional calculus.)  This evidently innocuous condition is in fact very rigid, as is codified in L\"owner's theorem.

\begin{theorem}[L\"owner 1934]\label{low}
Let $f:(a,b) \to \R$. $f$ is matrix monotone if and only if $f$ is real analytic on $(a,b)$ and analytically continues to the upper half plane in $\C$ as a map into the closed upper half plane.
\end{theorem}

For example, the functions $x$, $\log x$, $\sqrt{x}$, $\tan x$, and  $-x^{-1}$ are all matrix monotone on intervals in their domains, but $e^x$, $x^3$, and $\sec x$ are not. Note that matrix monotonicity is a geometric property; matrix monotonicity on a single interval implies matrix monotonicity on any interval where the function is real-valued in the real domain for analytic functions. L\"owner's theorem arises in many contexts, including mathematical physics \cite{wigner, wigner2}. Other applications are found, for example, in quantum data processing \cite{ahls}, wireless communications \cite{Jorswieck2007f, Boche2004e} and engineering \cite{alcober, anderson, osaka}.

Nevanlinna\cite{nev22, lax02} showed that all such functions on the unit interval are of the form
\[
f(x) = a + \int_{[-1,1]} \frac{x}{1 + tx} \, d\mu(t)
\]
for $a \in \R$ and $\mu$ a finite measure supported on $[-1,1]$. The Nevanlinna representation tells us exactly how to analytically continue a function to the upper half plane. 

Let $f:(a,b) \to \R$ be a function. We say that $f$ is \dfn{matrix convex} if 
\[ 
f\left(\frac{A + B}{2}\right) \leq \frac{f(A) +f(B)}{2}
\]
for all $A, B$ self-adjoint with spectrum in $(a,b)$. L\"owner's student Kraus proved the following theorem, which is ostensibly more technical, but demonstrates the same essential rigidity.

\begin{theorem}[Kraus 1937]\label{kraus}
Let $f:(-1,1) \to \R$. $f$ is matrix convex if and only if 
\[
f(x) = a + bx + \int_{[-1,1]} \frac{x^2}{1 + tx} \, d\mu(t)
\]
where $a, b \in \R$ and $\mu$ is a finite measure supported on $[-1,1]$. Note that all such functions analytically continue to the upper half plane.
\end{theorem} 

For example, $x^2$ is matrix convex, but $x^4$ is not. 

\subsection{Free noncommutative function theory}

Let $R$ be a real topological vector space. Define the \dfn{matrix universe over $R$}, denoted by $\M(R)$, by
\[
\M(R) = \bigcup_{n\in \mathbb{N}} M_n(\C) \otimes_\R R,
\]
where $M_n(\C)$ is the space of $n$ by $n$ matrices over $\C$. The space $\M(R)$ is endowed with the disjoint union topology. Given $V \subset \M(R)$, denote by $V_n$ the set $V \bigcap M_n(\C) \otimes R$. Define the \dfn{Hermitian matrix universe over $R$}, denoted by $\mathcal S(R)$, to be 
\[
\mathcal S(R) = \bigcup_{n \in \mathbb{N}} S_n(\C) \otimes_\R R,
\]
where $S_n(\C)$ denotes the space of $n$ by $n$ Hermitian matrices.

A set $G \subset \M(R)$ is defined to be a \dfn{(free) domain} if it satisfies the following axioms:
\begin{enumerate}
\item $X \oplus Y \in G \Leftrightarrow X, Y \in G$
\item $X \in G_n \Rightarrow U\ad X U \in G$ for all $n$ by $n$ unitaries $U$ over $\C$
\item $G_n$ is open for all $n$.
\end{enumerate}

Let $G \subset \M(R_1)$ be a free domain. We say a function $f: G \to \M(R_2)$ is a \dfn{free function} if
\begin{enumerate}
\item $f|_{G_n}$ maps into $\M(R_2)_n$,
\item $f(X \oplus Y) = f(X) \oplus f(Y)$,
\item $S\inv f(X) S = f(S\inv X S)$ for all $n$ by $n$ invertible $S$ over $\C$ such that $X, S\inv X S \in G_n$.
\end{enumerate}

If $R$ is a real operator system -- that is, a real subspace containing $1$ of self-adjoint elements in a $C\ad$-algebra - then for each $n$ there is a natural ordering on $S_n(\C) \otimes R$, since matrices over $R$ are elements of a larger $C\ad$-algebra. (The Choi-Effros Theorem \cite{choieffros} gives that any abstract Archimedean matrix ordering in a very general sense is equivalent to this situation. That is, this is the most general setup.) Given $A, B \in S_n(\C) \otimes R$, we say $A \leq B$ if $B - A$ is positive semidefinite as an element of $S_n(\C) \otimes R$.

Given $R_1$ and $R_2$ real operator systems and a domain $G \subseteq \mathcal S(R_1)$, say that a free function $f: G \to \mathcal S(R_2)$ is \dfn{matrix monotone} if 
\[
A \leq B \Rightarrow f(A) \leq f(B)\] 
whenever $A$ and $B$ have the same size. We say a domain $G\subseteq \mathcal S(R_1)$ is \dfn{convex} if each $G_n$ is convex. For a convex domain $G \subseteq \mathcal S(R_1)$, say that a free function $f: G \to \mathcal S(R_2)$ is \dfn{matrix convex} if
\[
f\left(\frac{A + B}{2}\right) \leq \frac{f(A) + f(B)}{2}
\] 
for all pairs $A, B \in G$ of the same size.

Define the \dfn{upper half plane} $\Pi(R) = \{X \in \M(R) | \im X > 0\}$, where $\im X = (X - X\ad)/2i$, and $A > B$ if the difference is strictly positive definite -- that is, the difference is self-adjoint and its spectrum is a subset of $(0,\infty)$. For a convex domain $G \subseteq \mathcal S(R)$, define the \dfn{tube over $G$} to be the set 
\[
T(G) = \{X + iY | X \in G \text{ and } Y = Y\ad\}.
\]

In several commuting variables, generalizations of L\"owner's theorem appear in \cite{amyloew, pascoelownote}. The proofs are technical and involved, and rely heavily on commutative Hilbert space techniques. The difficulty is a symptom of the fact that the variety of commuting tuples of matrices is full of holes -- that is, it is not convex and, thus, unnatural for understanding monotonicity. By contrast, the machinery of several complex variables is apparently much more natural in the noncommutative setting. Noncommutative analogues of L\"owner's theorem have previously been established in \cite{pastdfree, palfia}. The culmination of this work appears in \cite{pascoeopsys}, where the following theorem was proved in perhaps the highest level of generality that one should expect (although that proof relies on the commuting theorem in \cite{amyloew} and is thus ``unnatural'').

\begin{theorem}[Theorem 1.2, Pascoe \cite{pascoeopsys}]
Let $R_1$ and $R_2$ be closed real operator systems. Let $G \subseteq \mathcal{S}(R_1)$ be a convex free domain.  A function $f: G \to \mathcal{S}(R_2)$ is matrix monotone if and only if $f$ is real analytic on $G$ and analytically continues to $\Pi(R_1)$ as a map into $\cc{\Pi(R_2)}$.
\end{theorem}
We give a new proof of this result as Theorem \ref{newlow} using the ``royal road''.

We note two important examples of matrix monotone functions. The Schur complement $X_{11} - X_{12}X_{22}\inv X_{21}$ gives a matrix monotone function on the set $D \subset \mathcal{S}(S_2(\C))$, the space of block $2$ by $2$ self-adjoint matrices, where $X_{22}\inv$ is defined \cite{liuwang}. Another example is the matrix geometric mean, originating in mathematical physics \cite{pusz}, given by the formula 
$X_1^{1/2}(X_1^{-1/2}X_2 X_1^{-1/2})^{1/2} X_1^{1/2}$ defined on pairs of positive matrices in $\mathcal{S}(\R^2)$ \cite{lawsonlim, bhatiakar, ando94}.

Analogues of Kraus's theorem are less general. One example is the so-called ``butterfly realization'' developed in \cite{heltonbutterfly} for noncommutative rational functions, which captures the essence of the classical case. 

\begin{theorem}[Theorem 3.3, Helton, McCullough, Vinnikov \cite{heltonbutterfly}]
Let $r: G \subset \mathcal{S}(\R^d) \to \mathcal{S}(\R)$ denote a noncommutative rational function on a domain $G$ containing $0$. If $r$ is matrix convex near $0$, then $r$ has a realization of the form
\[
r(X) = r_0 + L(X) + \Lambda(X)\ad (1 - \Gamma(X))\inv \Lambda(X) 
\]
for a scalar $r_0$, a real linear function $L$, $\Lambda$ affine linear,  and $\Gamma(X) = \sum A_i \otimes X_i$ for self-adjoint matrices $A_i$.
\end{theorem}

We prove the butterfly realization holds for general matrix convex functions in Corollary \ref{abutterfly}. 

\subsection{The royal road theorem}

The main result of the paper is contained in Section \ref{sec:auto}. It establishes that any class of real free noncommutative functions which consist of locally bounded functions which are analytic on one-dimensional slices in a controlled way and closed under some basic algebraic and analytic procedures are automatically analytic. We call such a class of functions a \dfn{sovereign class}. The class of matrix monotone functions and the class of matrix convex functions are each sovereign classes. Once we know such functions are real analytic, algebraic and functional analytic techniques allow us to obtain nice formulas for these functions. The content of our main theorem, Theorem \ref{realsov}, states the following: 

\centerline{{\bf ``Any function in a sovereign class is real analytic''}.}

\subsection{Structure of the paper}

In Section 2, we discuss analytic continuation in the operator system setting. In Section \ref{sec:auto}, we describe the structure of the domain and function classes under consideration, the so-called sovereign functions, and show that matrix monotone and matrix convex functions are examples. We also prove the ``royal road'' theorem, the main engine of the machine under construction, which asserts that sovereign functions are automatically real analytic. In Section 4, we prove analogues of the classical L\"owner and Kraus realizations. In Section 5, we show that, in analogy with the classical case, we can deduce analytic continuations from the L\"owner and Kraus realizations using the machinery of automatic analyticity in classes of sovereign functions established in Section \ref{sec:auto}.
%
%
%
%
%
%
%
%
%
%
%
%

\section{Prelude: the quantitative wedge-of-the-edge theorem}

One of the key notions in the classical and several variable generalizations of the L\"owner and Kraus theorems is that of analytic continuation - that is, typically we are interested in extending functions from a ``real'' domain to some subset of a ``complex'' set. The edge-of-the-wedge theorem (proven by Bogoliubov and treated by Rudin in a series of lectures \cite{rudeow}) is useful in showing that such a continuation exists. Extremely flexible generalizations of this result to several variables have appeared in \cite{pascoecmb, pascoeblmswedge}. The key lemma from \cite{pascoecmb} follows, which we will need to generate quantitative bounds. In this section, we prove a version of the wedge-of-the-edge theorem in the operator system setting.

\begin{lemma}[Lemma 2.3, Pascoe \cite{pascoecmb}]
Fix $n.$ Fix $p>0.$ There are constants $C, K >0$ such that for every $S \subseteq [0,1]^n$ of measure greater than $p,$ and homogeneous polynomial $h$ of degree $d$ in $n$ variables which is bounded
by $1$ on $S$,
$|h(z)| \leq KC^d\|z\|_{\infty}^d.$
\end{lemma}
Such an assertion seems foolish, but it is essentially the product of Lagrange interpolation, blind faith, and elbow grease. 

Let $R_1, R_2$ be vector spaces. Define a (noncommutative) \dfn{generalized homogeneous polynomial of degree $d$} to be a (free) function on $R_1$ such that the restriction to
any finite dimensional space is an $R_2$-valued (noncommutative) homogeneous polynomial
of degree $d$.
\begin{lemma}
There are universal constants $C, K >0$  satisfying the following.
Let $R$ be an operator system.
Let $W$ be the set of positive contractions in $R$ ($\mathcal{S}(R)$ in the noncommutative case).
Let $h$ be a (noncommutative) generalized homogeneous polynomial of degree $d$ which is norm bounded by $1$ on $W.$
Then,
$\|h(Z)\| \leq KC^d\|Z\|^d.$
\end{lemma}
\begin{proof}
	It is enough to prove the claim when $\|Z\|=1,$ as both sides are homogeneous of degree $d$.
	Write $Z  = A - B+ iC - iD$ for positive $A, B, C, D$, where the norms of $A, B, C, D$ are less than 
	$2\|Z\|.$
	The function of four variables $f(x_1,x_2,x_3,x_4)=h((x_1A+x_2B+x_3C+x_4D)/8)$ satisfies 
	the preceeding lemma when composed with any norm $1$ linear functional for $S = [0,1]^4$, so, by the Hahn-Banach theorem, $\|h(Z)\|= \|f(8,8,8,8)\| \leq KC^d 8^d.$ 
\end{proof}

Define the \dfn{complex ball around $X$ of radius $\ep$,} denoted $B_\C(X, \ep),$ to be 
\[B_\C(X, \ep)=\bigcup_m \{Y\in \mathcal{M}(R)_{mn}|\norm{X^{\oplus m}-Y}<\ep\}.\]
Define the \dfn{real ball around $X$ of radius $\ep$,} denoted $B_\R(X, \ep)$ to be 
\[B_\R(X, \ep)=\bigcup_m \{Y\in \mathcal{S}(R)_{mn}|\norm{X^{\oplus m}-Y}<\ep\}.\]

The following corollary follows immediately from the preceding lemma.
\begin{corollary}[The quantitative wedge-of-the-edge theorem]
There are universal constants $\delta, \ep >0$  satisfying the following.
Let $R$ be an operator system.
Let $W$ be the set of positive contractions in $R$ ($\mathcal{S}(R)$ in the noncommutative case).
Let $h_d$ be a sequence of (noncommutative) generalized homogeneous polynomials of degree $d$ such that
$\sum \|h_d(X)\|$ is bounded by $1$ on $W.$
The formula $\sum h_d(Z)$ defines a (noncommutative) analytic function on $B_{\C}(0,\delta)$ which is bounded by $\ep.$
\end{corollary}

\section{ Automatic analyticity in sovereign classes}\label{sec:auto}

Let $G \subseteq \M(R)$. We define the \dfn{coordinatization} of $G$, denoted $G^{(n)}$, to be the natural inclusion of $(G_{mn})_{m=1}^\infty$ into $\M(R \otimes M_n(\C))$. 

Let a \dfn{dominion} $\GG$ be a class of domains satisfying:
\begin{description}
\item[Translation invariance] For all $ G \in \GG$ and $r \in R$, $G + r \in \GG$.
\item[Closure under intersection] For all $G, H \in \GG$, $G \cap H \in \GG$. 
\item[Closure under coordinatization] If $G \in \GG$, then $G^{(n)} \in \GG.$
\item[Locality] 
Let $G \in \GG.$ For any $X \in G_1,$ there is an $\epsilon >0$ such that $B_\R(X,\ep) \subseteq G$ and $B_\R(X,\ep)\in \GG.$
\item[Scale invariance] If $t > 0$ and $G \in \GG$, $tG \in \GG$. 
\end{description}

An example of a dominion is the class of all matrix convex sets, which we denote $\conv$.

A \dfn{sovereign class} is a class of functions $\sov$ on domains contained in a dominion $\GG$ satisfying:
\begin{description}
\item[Functions] For all $G \in \GG$, $\sov(G) \subseteq \funct(G)$, where $\sov(G)$ denotes the 
functions in $\sov$ on the domain $G$ and $\funct(G)$ denotes the class of free functions on $G$.
\item[Local boundedness] Each $f \in \sov(G)$ is locally bounded and measurable on finite dimensional affine subspaces on each level.
\item[Closure under localization] If $f \in \sov(G)$ and $H \subseteq G$ then $f|_{H} \in \sov(H)$.
\item[Closure under coordinatization] If $f \in \sov(G)$, then $f^{(n)} \in \sov(G^{(n)})$.
\item[Closure under convolution] The set of functions $\sov(G)$  taking values in $\mathcal{S}(R)$ is convex and closed under pointwise weak limits. 
\item[One-variable knowledge] If $A \leq B$ then 
\[
f_{\cc{AB}}(t) := f\left(\frac{1 - t}{2}A + \frac{1 + t}{2}B \right)
\] 
analytically continues to $\D$ as a function of $t$.
\item[Control] There is a map $\gamma$ taking each pair $(X, f)$ to a non-negative number satisfying:
\begin{enumerate}
\item {For each $\ep>0$ there is a universal constant
$c(\ep)$ such that $\inf_{X\in B_\R(X_0,\ep)_1 
} \gamma(X,f) \leq c(\ep)\|f\|_{B_\R(X_0,\ep)_1}.$}
\item There is a universal positive valued function $e$ on $\R^+$ satisfying the following. Write $f_{\cc{AB}}(t) = \sum a_n t^n$. Then 
\[
\norm{a_n} \leq \gamma(X, f)e(\norm{B - A}).
\]
Note that, if the class is closed under composition with positive, norm one, linear functionals, and $\gamma(X,\lambda(f)) \leq \gamma(X,f),$  it is sufficient to check this when $R_2=\R$ by the Hahn-Banach theorem.
\item If $H \subseteq G$ and $X \in H$ then $\gamma(X, f|_{H}) = \gamma(X, f)$.
\item $\gamma(X, f) = \gamma(X^{\oplus N}, f)$.
\item $\gamma(X^{(n)},f^{(n)}) = \gamma(X, f)$.
\end{enumerate} 
\end{description}

We consider two specific sovereign classes: monotone functions, and convex functions on the dominion $\conv$.

We define the \dfn{positive-orthant norm of the $n$-th derivative at $X$}, denoted $\|D^nf(X)\|$, to be
\[
\|D^nf(X)\|=\sup_{\|H\|=1, H>0, m} \|D^nf(X^{\oplus m})[H]\|,
\] where $D^nf(X)[H] = \frac{d^n}{dt^n} f(X+tH).$

\begin{proposition}
The matrix monotone functions on domains in $\conv$ are a sovereign class.
\end{proposition}

\begin{proof}
Monotone functions are functions. To see local boundedness, note that $f(X + 1)$ and $f(X - 1)$ bound $f(X + H)$ for all $\norm{H} < 1$. That is, as 
\[
X - 1 \leq X + H \leq  X + 1,
\]
monotonicity implies 
\[
f(X - 1) \leq f(X + H) \leq  f(X + 1).
\]
The restriction of a monotone function to a convex set remains a monotone function. Likewise, coordinatization preserves monotonicity. That the monotone functions are closed under convolution follows from the fact that the defining inequality for monotonicity is linear. Monotone functions analytically continue to the upper half plane and lower half plane, and thus the disk $\D$, whenever $(-1,1)$ is in the domain as is the case for $f(\frac{1-t}{2}A + \frac{1 + t}2 B)$.

Fix $\ep >0$. Suppose that $B_{\R}(X, \ep)$ is contained in the domain of $f$. Without loss of generality, $0 = X$. Fix $H\geq 0$ in $B_{\R}(0,\ep)$. So $f(zH)$ has a Nevanlinna type representation given by
\begin{align*}
f(zH) &= a_0 + \int_{[-1,1]} \frac{z}{tz + 1} \, d\mu(t) \\
&= a_0 + z \sum_{i=0}^\infty \int t^i z^i \, d\mu(t).
\end{align*}
Note that this shows that $|a_n| = \int |t|^{n-1}  \, d\mu(t)\leq \int  \, d\mu(t)  = a_1 $ for $n\geq 1.$
Moreover,
\[
f(zH) - f(-zH) = 2 z \sum_{i=0}^\infty \int z^{2i} t^{2i} \, d\mu.
\]
This shows that
\[
\norm{Df(0)[H]} \leq \norm{f}_{B_\R(0,\ep)}. 
\]
Therefore, 
\[
\norm{Df(0)[H]} \leq \frac{1}{\ep} \norm{f}_{B_\R(0,\ep)}.
\]
Now, a control function is given by the formula
\[
\gamma(X,f) = \norm{f(X)} + \norm{Df(X)}, 
\]
which is bounded by $(1 + \frac{1}{\ep}) \norm{f}_{B_\R(0, \ep)}$.
 
\end{proof}
 
\begin{proposition}
The locally bounded matrix convex functions on domains in $\conv$ are a sovereign class.
\end{proposition}

\begin{proof}
Convex functions are functions. The restriction of a convex function to a subdomain remains convex. The coordinatization of a convex function is convex. Closure under convolution follows from the fact that the defining inequality for convexity is linear. By the Kraus theorem, these functions satisfy one variable knowledge.

Fix $\ep >0$. Suppose that $B_{\R}(X, \ep)$ is contained in the domain of $f$. Without loss of generality, $0 = X$. Fix $H$ in $B_{\R}(0,\ep)$.
 The function $f(zH)$ has a Kraus type representation
\[
f(zH) = a + bz + \int_{[-1,1]} \frac{z^2}{tz + 1} \, d\mu(t).
\]
We have
\[
f(zH) + f(-zH) = 2\left(a + z^2 \sum \int z^{2i} t^{2i} \, d\mu(t)\right).
\]
Note that this shows that $|a_n| = \int |t|^{n-2}  \, d\mu(t)\leq \int  \, d\mu(t)  = a_2 $ for $n\geq 2.$
This shows that
\[
\norm{D^2f(0)[H]} \leq \norm{f}_{B_\R(0, \ep)}.
\]
Therefore $\norm{D^2f(0)} \leq \frac{1}{\ep^2}\norm{f}_{B_\R(0, \ep)}$.
Denote $M = \norm{f}_{B_\R(0, \ep)}$. Now consider $\abs{f(zH) - a}$. This is bounded by $2M$. Therefore
\[
\norm{bz} - \norm{z^2 \sum \int z^i t^i \, d\mu(t)} \leq 2M 
\]
which gives
\begin{align*}
\norm{b} &\leq \norm{z \sum \int z^i t^i \, d\mu(t)} + \frac{2M}{\abs{z}} \\
&\leq \norm{z \sum \int z^i \, d\mu(t)} + \frac{2M}{\abs{z}} \\
&\leq \abs{\frac{z}{1-z}}\norm{D^2f(0)[H]} + \frac{2M}{\abs{z}} \\
&\leq \abs{\frac{z}{1-z}}M + \frac{2M}{\abs{z}} \\
\end{align*}
Pick $z = \frac{1}{2}$. Then 
\[
\norm{Df(0)[H]} = \norm{b} \leq 5M.
\]
Therefore,
\[
\norm{Df(0)} \leq \frac{1}{\ep} 5M.
\]
A control function $\gamma$ is given by
\[
\gamma(X,f) = \norm{f(X)} + \norm{Df(X)} + \norm{D^2f(X)}.
\]
\end{proof}

We note that any matrix convex function on a finite dimensional space will be continuous and thus locally bounded. Some sort of topological restriction, such as local boundedness, is necessary, as arbitrary linear maps on any operator system are not necessarily bounded but are definitely convex, as all linear functions are convex.

\begin{lemma}\label{realeachlevel}
Any function in a sovereign class is real analytic at each level on each finite dimensional affine subspace containing the identity direction.
\end{lemma}
\begin{proof}
Without loss of generality, we will assume $R_1$ is finite dimensional. Fix $X \in G_n$. Without loss of generality, $0 = X \in G_1$ by closure under coordinatization and translation. Also without loss of generality, assume that $B_\R(X, 2) \subset G$. Let $\ph$ be a compactly supported positive smooth function on $R_1$. Define $\ph_\alpha(x) = \frac{1}{\alpha}\ph\left(\frac{1}{\alpha}x\right)$. Consider 
\[
f_\alpha(Y) = (\ph_\alpha * f)(Y) = \int_{R_1} f(Y - r)\ph_\alpha(r).
\]
As a sovereign class of functions is closed under convolution, for small enough $\alpha$, the function $f_\alpha$
will be in the sovereign class of $B_\R(X,2-\ep)$ for any fixed $\ep.$
Choose $Y \in B_\R(X, \delta/2)_1$ such that $\gamma(Y,f) \leq 2c(\delta)\|f_\alpha\|_{B_{\R}(X,\delta)_1}$ (which exists by the definition of our control function).
where $\delta<1$ comes from the quantitative wedge-of-the-edge theorem.
Note that $f_\alpha|_{B_\R(X,2-\ep)_1}$ is smooth at $Y$ and by the one variable knowledge $f_\alpha(Y+Z) = \sum h_d(Z)$
on positive contractions in $R_1$. By the control properties, we see that $\sum \|h_d(Z)\|$ is bounded by some $M$ on the positive contractions as we have uniform bounds on the Taylor coefficients, and therefore by the quantitative wedge-of-the-edge theorem,
$f_\alpha$ continues to a function bounded by $M\ep$ on $B_{\C}(Y,\delta)_1.$ Therefore, $f$ extends analytically and is bounded by $M\ep$ on $B_{\C}(Y,\delta)_1$ by a normal families argument.
As $B_\C(X, \delta/2)_1 \subseteq B_{\C}(Y,\delta)_1,$ we are done.
\end{proof}

Let $G \subset \mathcal{S}(R_1)$ be a real domain. Let $f: G \to \M(R_2)$. Fix $X \in G_n$. $f$ is \dfn{real analytic at $X$} if there is a $\delta > 0$ such that for any choice of $H_i$, the induced free function $f(X + \sum H_i t_i) = \sum a_\alpha t^\alpha$ for all $\norm{\sum H_i t_i } < \delta$. Equivalently, $f^{(n)}(X + Y) = \sum h_j(Y)$ is uniformly convergent on $B_\C(X, \delta)$ for noncommutative generalized homogeneous polynomials $h_j$. 

We adopt the (by now standard) Helton convention of suppressing tensor notation for products of operators $A$ and noncommutative indeterminants $x_i$; that is, we write $Ax_i$ for $A \otimes x_i$.

\begin{theorem}[The royal road theorem]\label{realsov}
	Any function in a sovereign class is real analytic. 
\end{theorem}

\begin{proof}
Fix $X \in G_n$. Without loss of generality, $0 = X \in G_1$ by closure under coordinatization and translation. Also without loss of generality, assume that $B_\R(X, 1) \subset G$. 
Therefore, since $f$ is real analytic at each level by Lemma \ref{realeachlevel},
$f(X) = \sum h_d(X)$ for some noncommutative homogenous generalized polynomials $h_d$ on the set of positive contractions
in $\mathcal{S}(R_1).$ Moreover, the series is bounded on smaller balls by the control properties, as we have uniform bounds on the Taylor coefficients on each positively oriented one dimensional slice. Thus, by the noncommutative quantitative wedge-of-the-edge theorem, the function $f$ must be bounded and analytic on $B_\C(X, \delta)$ for some $\delta >0$. This establishes the claim.
\end{proof}

\section{Realizations and the Kraus theorem}
In the following section, we will usually assume that $R_1 = \R^d$ and always that $R_2$ is contained in some concrete
$\BK.$ We will frequently use free noncommutative  power series of the form
\[f(Z) = \sum_{\alpha} c_\alpha Z^\alpha,\]
where $\alpha$ runs over all words in the formal noncommuting letters $x_1, \ldots, x_d,$ where the empty word will be denoted by $1.$ (Words are the natural multi-indices in the noncommutative setting.)
Various series representations can be derived via model-realization theory \cite{vvw12,bgm05,bgm06,agmc_gh, bmv18} with many results for the homogenous expansion.

\subsection{Monotonicity}

The following lemma is essentially \cite[Theorem 4.16]{pastdfree} lifted to the multi-dimensional output setting. 

\begin{lemma}
Suppose that $f(X) = \sum c_\alpha X^\alpha$ is analytic on $B_\C(0,1)\subseteq \mathcal{S}(\R^d)$ and that $f$ is matrix monotone. For each $i = 1,\ldots, d$,
the $x_k$-localizing matrices (with operator entries) satisfy 
\[
C_i = \left[c_{\beta\ad x_i \alpha}\right]_{\alpha, \beta} \geq 0
\]
where $\alpha, \beta$ range over all monomials.
\end{lemma}
\begin{proof}
	Note
\[
		Df(X)[H] = \sum_{\alpha,\beta, i} c_{\beta^*x_i\alpha} X^{\beta^*}H_iX^\alpha.
\]
	We can write 
		\[Df(X)[H] = \sum_i (I_\KK \otimes K_X)^*(C_i \otimes H_i) (I_\KK \otimes K_X)\]
	where $K_X$ is the vector-valued free function $(X^\alpha)_\alpha$.
	Taking $H_i = vv^*,$ and the rest zero then defining a vector-valued function $K_X^v(w)= (I_\KK \otimes (v^*X^\alpha)_\alpha )w$, we see, by monotonicity, that
	 	$K_X^v(w)^* C_i K_X^v(w)\geq 0.$
	So it suffices to show that the range of $K_X^v(w)=(I_\h \otimes (v^*X^\alpha)_\alpha) w$ is dense.
	It is an elementary exercise to show that their span is dense, say by viewing the ambient setting
	as a kind of reproducing kernel Hilbert space. 
	(See, for example, \cite[Proposition 3.9]{pastdfree}.)
	Therefore, it is sufficient to show that the range is closed under taking sums.
	One checks that 
	\[K_{X_1}^{v_1}(w_1) + K_{X_2}^{v_2}(w_2) = K_{X_1\oplus X_2}^{v_1 \oplus v_2}(w_1 \oplus w_2).\]
	So, we are done.

\end{proof}

\begin{theorem}\label{monoreal}
Let $f$ be a matrix monotone function whose power series conveges absolutely and
 uniformly on $B_\C(0, 1+\ep)\subseteq \mathcal{S}(\R^d)$. Let $\h_i$ be the Hilbert space equipped with the inner product 
\[
\ip{\alpha\otimes v}{\beta\otimes w}_{\h_i} = w^*c_{\beta\ad x_i \alpha} v.
\]
Let $\h = \oplus \h_i$ and $P_i$ be the projection onto $\h_i$. Note that 
\[
\ip{\alpha\otimes v}{ \beta \otimes w}_\h = \sum_i w^*c_{\beta\ad x_i \alpha}v.
\]
Define $A: \h \to \h$ by
\[
A(\alpha\otimes v) = \sum_i (x_i \alpha)\otimes w.
\]
Let Q be the map taking $k\in \KK$ to $1 \otimes k \in \h$.
The operator $A$ is a bounded self-adjoint contraction on $\h$, and
 
\[
f(Z) = a_0 + Q^*(A - \sum_i P_i Z_i\inv)\inv Q.
\]
\end{theorem}

\begin{proof}

To see that $A$ is self-adjoint, compute
\begin{align*}
\ip{A (\alpha\otimes v)}{\beta\otimes w} &= \ip{\sum_i x_i \alpha\otimes v}{\beta\otimes w} \\
&= \sum_i \ip{x_i \alpha\otimes v}{\beta\otimes w} \\
&= \sum_i \sum_j w^*c_{\beta\ad x_j x_i \alpha}v \\
&= \sum_j \sum_i w^*c_{\beta\ad x_j x_i \alpha}v \\
&= \sum_j \ip{\alpha\otimes v}{x_j \beta\otimes w} \\
&= \ip{\alpha\otimes v}{\sum_j x_j \beta\otimes w} \\
&= \ip{\alpha\otimes v}{A \beta\otimes w}.
\end{align*}
To see that $A$ is contractive, we will use the fact that 
\[
\norm{A} = \rho(A) = \sup_{\norm{v}=1} \sup_\alpha \liminf_{n \to \infty} \norm{A^n \alpha \otimes v}^{1/n}.
\]
Write
\begin{align*}
\norm{A^n (\alpha\otimes v)}^2 &= \ip{A^n (\alpha\otimes v)}{A^n (\alpha\otimes v)} \\
&=\ip{(\sum x_i)^n \alpha\otimes v}{(\sum x_i)^n \alpha\otimes v} \\
&= \sum_{\abs{\omega} = 2n+1} v^*c_{\alpha\ad \omega \alpha}v \\
&\leq \sum_\omega \abs{v^*c_\omega v}.
\end{align*}
The power series converges uniformly and absolutely on the ball of radius 1, and thus the coefficients are uniformly bounded. This implies that $\rho(A) \leq 1$.

We will now establish that $A P_j (\alpha\otimes v) = x_i \alpha \otimes v$.
\begin{align*}
\ip{A P_j (\alpha\otimes v)}{\beta\otimes w} &= \ip{P_j \alpha \otimes v}{A \beta \otimes w} \\
&= \ip{P_j (\alpha\otimes v)}{ \sum_i x_i \beta\otimes w} \\
&= \sum_i \ip{P_j (\alpha\otimes v)}{x_i \beta\otimes w} \\
&= \sum_i w^*c_{\beta\ad x_i x_j \alpha}v \\
&= \ip{x_j \alpha \otimes v}{\beta\otimes w}.
\end{align*} 

We now compute the realization to see that it agrees with $f$.
\begin{align*}
w^*f(Z)v &= \sum_\alpha w^*c_{\alpha}v Z^\alpha \\
&= w^*c_1v + \sum_i \sum_{\alpha} w^*c_{x_i \alpha}v Z^{x_i \alpha} \\
&= w^*c_1v + \sum_i \sum_\alpha \ip{P_i (\alpha\otimes v)}{1\otimes w} Z^{x_i \alpha}\\
&= w^*c_1v + \sum_i \sum_\alpha \ip{P_i (AP)^\alpha (1\otimes v)}{1\otimes w} Z^{x_i \alpha}\\
&= w^*c_1v + \ip{(A - \sum_i P_i Z_i\inv)\inv (1\otimes v)}{(1\otimes w)}_{\h} \\
& = w^*c_1v + w^*Q^*(A - \sum_i P_i Z_i\inv)\inv Qv.
\end{align*}
\end{proof}
We note that, in general, noncommutative Pick functions have representations of the form 
$a_0 + E((A-Z^{-1})^{-1})$ whenever they are analytic on a neighborhood of $0$ and $R_1$ is a $C^*$-algebra,
where $E$ is a completely positive map \cite{will13,pastdcauchy}. The theory of such ``Cauchy transforms'' is well understood in the context of free probability \cite{anw14,will15}.

\subsection{Convexity}

\begin{lemma}
Suppose that $f$ is analytic on $B_\C(0,1)\subseteq \mathcal{S}(\R^d)$ and that $f$ is matrix convex. 
The block matrix (with operator entries),
\[
C=\left[c_{\beta\ad\alpha}\right]_{\alpha,\beta} \geq 0
\]
where $\alpha, \beta$ range over all monomials of degree greater than or equal to $1$.
\end{lemma}
\begin{proof}
	Note
\[
		D^2f(X)[H] = \sum_{\alpha,\beta,\gamma, i,j}
		c_{\beta^*x_i\gamma x_j\alpha} X^{\beta^*}H_i X^\gamma  H_j X^\alpha\geq 0. 
\]
	Under the subsitution
	\[X \mapsto \bbm X & 0 \\ 0 & 0 \ebm,  H\mapsto \bbm 0 & Xv \\ (Xv)^* & 0 \ebm, \]
	and taking the $1,1$ entry of the above relation,
	we see that 
		\[\sum_{\alpha,\beta, i,j}
		c_{\beta^*x_ix_j\alpha} X^{\beta^*x_i}v v^* X^{x_j \alpha}\geq 0.\]
	Therefore, considering the function $K^v_X(w)= (I_\KK \otimes (v^*X^\alpha)_\alpha) w$
	we see again that the range is dense, so we are done.
\end{proof}

The following theorem is related to the ``butterfly realization'' for noncommutative rational functions in \cite{heltonbutterfly}.

\begin{theorem}\label{butterfly}
Let $f$ be a matrix convex function whose power series conveges absolutely and uniformly on  $B_\C(0,1+\ep)\subseteq \mathcal{S}(\R^d)$. Let $\h$ be a Hilbert space equipped with the inner product
\[
\ip{\alpha\otimes v}{\beta \otimes w} = w^*c_{\beta\ad\alpha}v
\]
where $\alpha, \beta$ range over all monomials with degree greater than or equal to $1$ and $v, w$ range over $\KK.$
Define the self-adjoint operators $T_i$ by
\[
T_i (\alpha\otimes v) = x_i \alpha \otimes v.
\]
Let $Q_i$ be the map taking $v\in \KK$ to $x_i \otimes v \in \h.$
The operators $T_i$ are contractions and
\[
f(Z) = a_0 + L(Z) + (\sum Q_i Z_i\ad)^* (I - \sum T_i Z_i)\inv (\sum Q_i Z_i)
\]
for some choice of $a_0$ and continuous linear function $L.$
\end{theorem}

\begin{proof}
That the realization formula is equivalent to the function when the $T_i$ are contractions is a standard algebraic manipulation.  The nontrivial part of the proof, then, is to show that the $T_i$ are contractive.

We proceed by a spectral radius argument as before. 
\begin{align*}
\norm{T_i^n (\alpha \otimes v)}^2 &= \ip{T_i^n \alpha\otimes v}{T_i^n \alpha \otimes v} \\ 
&= v^*c_{\alpha\ad x_i^{2n} \alpha}v.
\end{align*}
The coefficients must be uniformly bounded, as the power series converges uniformly and absolutely on the ball of radius $1$. This completes the proof.
\end{proof}

We remark that the construction of the realization is essentially canonical, and therefore must have maximal domain, (as opposed to our {\it a priori} assumption of a ball) as the realization at any point can be used to determine the realization at any other point on connected sets. (That is, a matrix convex function with a realization as above defined on a convex domain $G$ must have $I- \sum T_i Z_i$ positive for all $Z \in G.$) Moreover, by a limiting argument, a matrix convex function on a domain containing $0$ over a general operator system should be of the form:
\[
f(Z) = a_0 + L(Z) + \Lambda(Z\ad)^* (I -  \Gamma(Z))\inv \Lambda(Z)
\]
where $\Lambda:R_1 \rightarrow \mathcal{B}(\KK,\h)$ and $\Gamma: R_1 \rightarrow \BH$ are linear maps.
The boundedness of $\Lambda$ follows from the continuity of the second derivative, the continuity of $\Gamma$ follows from the fact that the spectral radius is bounded, essentially the same argument as before. 
That is, we have the following corollary.
\begin{corollary}[A noncommutative Kraus theorem]\label{abutterfly}
Let $R_1, R_2$ be real operator systems.
Let $G\subseteq \mathcal{S}(R_1)$ be a convex domain.
Let $f:G \rightarrow \mathcal{S}(R_2)$ be a locally bounded free function on a convex domain $G\subseteq \mathcal{S}(R_1)$ with $B \in G_1.$
The function $f$ is matrix convex if and only if 
\[
f(Z+B) = a_0 + L(Z) + \Lambda(Z\ad)^* (I -  \Gamma(Z))\inv \Lambda(Z)
\]
where $\h$ is a Hilbert space, $L:R_1\rightarrow \BK$, $\Lambda:R_1 \rightarrow \mathcal{B}(\KK,\h)$ and $\Gamma: R_1 \rightarrow \BH$ are completely bounded linear maps, where $L$ and $\Gamma$ are self-adjoint valued.
\end{corollary}
\begin{proof}
Without loss of generality $B=0,$ $f(0)=0$ and $Df(0) = 0.$ Moreover,
we assume $f$ has a uniformly convergent homogeneous power series on the unit ball, which exists by real analyticity.
 
Let $\mathcal R$ denote the collection of finite operator system subspaces of $R_1.$

Fix $R \in \mathcal{R}$. Pick a basis $r_1, \ldots, r_n.$
Consider the induced function $g(X)=f(\sum r_i X_i).$
We see that
$$g(Z) = (\sum Q_i Z_i\ad)^* (I - \sum T_i Z_i)\inv (\sum Q_i Z_i).$$
Call the representing Hilbert space $\h_R.$
Now,
$f|_R(Z) = \Lambda_R(Z\ad) (I - \Gamma_R(Z))\inv \Lambda_R(Z).$
Taking the second derivative, we get
\begin{align*}
&\Lambda_R(H\ad)^* (I - \Gamma_R(Z))\inv \Lambda_R(H) + \\
&\Lambda_R(Z\ad)^* (I - \Gamma_R(Z))\inv\Gamma_R(H)(I - \Gamma_R(Z))\inv \Lambda_R(H) + \\
&\Lambda_R(H\ad)^* (I - \Gamma_R(Z))\inv\Gamma_R(H)(I - \Gamma_R(Z))\inv \Lambda_R(Z) +\\
&\Lambda_R(Z\ad)^* (I - \Gamma_R(Z))\inv\Gamma_R(H)(I - \Gamma_R(Z))\inv\Gamma_R(H)(I - \Gamma_R(Z))\inv \Lambda_R(Z).&
\end{align*}

Under the substitution
	\[Z \mapsto \bbm 0 & 0 \\ 0 & Z \ebm,  H\mapsto \bbm 0 & H \\ H & 0 \ebm, \]
taking the $1,1$ entry we get 
$$\Lambda_R(H) (I - \Gamma_R(Z))\inv \Lambda_R(H).$$
The geometric expansion of this formula converges uniformly and absolutely.
Therefore for contractions, $\Gamma_R(Z)^{n}\Lambda_R(H)$ is eventually contractive. Now, taking $Z$ to be a strictly block upper triangular matrix with $Z_1, \ldots, Z_n \in B_\R(0,1)$ on the upper diagonal, we see that $\Gamma_R(Z_1)\Gamma_R(Z_2)\ldots\Gamma_R(Z_n)\Lambda_R(H)$ must be contractive for $n$ large enough, and therefore the joint spectral radius of the set $\{\Gamma_R(Z)|Z\in B_{\R}(0,1)_m\}$ is less than or equal to $1$ for each $m.$

By canonicity of the construction, if $R \subseteq S$, $\h_R$ embeds into $\h_S$ (for example we could have extended the basis we chose for $R$ in our original construction to a basis for $S$.) Moreover
$\Lambda_S|_R = \Lambda_R$ under this identification and $\Gamma_S|_R  = \Gamma_R \oplus J_{SR}$
for some linear map $J_{SR}.$
So, ordering the sets in $\mathcal{R}$ under inclusion, we can take a direct limit to obtain $\Gamma, \Lambda$ as desired. 
\end{proof}

\section{L\"owner and Kraus type continuation theorems}

\begin{theorem}\label{newlow}
Let $R_1, R_2$ be real operator systems.
Let $G\subseteq \mathcal{S}(R_1)$ be a convex domain. A free function  $f:G \rightarrow \mathcal{S}(R_2)$ is matrix monotone if and only if it analytically continues to the upper half plane.
\end{theorem}

\begin{proof}
We essentially follow \cite{pascoeopsys}, except we need not appeal to the perhaps technically daunting Agler, McCarthy, and Young theorem \cite{amyloew}. Note that it is enough to show that $f$ analytically continues at each level to a Pick function -- that is an analytic function from $\Pi(R_1)_1$ to $\cc{\Pi(R_2)}_1$ - and therefore, by coordinatization, it is enough to show that this occurs at level $1$. Moreover, it suffices to consider the case of finite dimensional $R_1$. Moreover, we can assume $0$ is in $G$. 

The function $f$ will analytically continue to a Pick function if and only if $\lambda \circ f$ analytically continues to a Pick function for all positive linear functionals $\la$ on $R_2$. Therefore, it is enough to consider the case where $R_2$ is one dimensional.

Pick $Z \in \Pi(R_1)_1$. Pick $H_1, \ldots, H_n > 0$ such that there is a point $(z_1, \ldots, z_n) \in \Pi(\R^n)_1$ with $Z = \sum H_i z_i$ and the $H_i$ span $R_1$. Now, $f(\sum H_i x_i)$ is a matrix monotone function of $x$ and therefore analytically continues to the upper half plane $\Pi(\R^n)_1$ by the realization formula in Theorem \ref{monoreal}, which pulls back to $\Pi(R_1)_1$. (Note, as we choose additional $H_i$, we exhaust more and more of $\Pi(R_1)_1$.) 

\end{proof}

\begin{theorem}
Let $R_1, R_2$ be real operator systems.
Let $G\subseteq \mathcal{S}(R_1)$ be a convex domain. If a free function $f:G \rightarrow \mathcal{S}(R_2)$ is matrix convex and locally bounded then $f$ analytically continues to the tube
\[
T(G) = \{X + iY| X \in G \text{ and } Y = Y\ad \}.
\]
\end{theorem}

\begin{proof}
Let $Z \in T(G)$. Without loss of generality, $Z \in T(G)_1$. We will show that $f$ is bounded on a noncommutative ball around $Z$. 

First, write $Z =X+ iY$. Without loss of generality, $X = 0$ and $f$ is bounded and analytic on $B_\C(0, 1+ \ep)$. Pick $W \in B_\C(0,1).$ By the realization formula in Theorem \ref{abutterfly}, 
\[
f(Z) = a_0 + L(Z) + \Lambda(Z\ad)^* (I -  \Gamma(Z))\inv \Lambda(Z).
\]
Therefore,
\begin{align*}
\norm{f(Z + W)} \leq \norm{a_0}+\norm{L}\|Z+W\|+\frac{1}{\ep}\|\Lambda\|^2\norm{Z+W}^2.
\end{align*}
This shows that $f$ analytically continues to a neighborhood of $Z$, which establishes the claim.
\end{proof}

\bibliography{references}
\bibliographystyle{plain}

\printindex

\end{document}